\documentclass[a4paper,12pt]{amsart}
\usepackage[utf8]{inputenc}
\usepackage[T1]{fontenc}
\usepackage[english]{babel}
\usepackage{amssymb,amsthm,times}
\usepackage{xcolor}
\usepackage{hyperref}
\usepackage{cancel}
\usepackage{graphicx}
\usepackage{float}
\usepackage{epstopdf}

\newcommand{\NormOnly}{\| \cdot \|}
\newcommand{\lin}{\mathop{\rm span}\nolimits}
\newcommand{\co}{\operatorname{co}}
\newcommand{\aco}{\operatorname{aco}}

\newcommand{\ws}{weak\mbox{$^*$}}

\newcommand{\R}{\mathbb R}
\newcommand{\N}{\mathbb{N}}
\newcommand{\E}{\mathbb{E}}
\newcommand{\PP}{\mathbb{P}}
\newcommand{\LL}{\mathbb{L}}
\newcommand{\dom}{\operatorname{dom}}

\newtheorem{thm}{Theorem}

\newtheorem{pro}[thm]{Proposition}

\newtheorem{cor}[thm]{Corollary}
\newtheorem{que}[thm]{Question}
\date{}

\title{\bf One-side  James' compactness Theorem }
\author{B. Cascales}
\address{Departamento de Matem\'{a}ticas, Universidad de Murcia,
30100 Espinardo (Murcia), Spain} \email{beca@um.es}

\author{J. Orihuela}
\address{Departamento de Matem\'{a}ticas, Universidad de Murcia,
30100 Espinardo (Murcia), Spain} \email{joseori@um.es}

\author{A. Pérez}
\address{Departamento de Matem\'{a}ticas, Universidad de Murcia,
30100 Espinardo (Murcia), Spain} \email{antonio.perez7@um.es}

\thanks{Partially supported by Ministerio de Economía y Competitividad and FEDER project MTM2011-25377. The third author was supported by a PhD fellowship of ``La Caixa Foundation''.
}



\begin{document}

\maketitle
\begin{abstract}
We present some extensions of classical results that involve elements of the dual of Banach spaces, such as Bishop-Phelp's theorem and James' compactness theorem, but restricting to sets of functionals determined by geometrical properties.  The main result, which answers a question posed by F. Delbaen, is the following:  
{\it Let $E$ be a Banach space such that $(B_{E^{\ast}}, \omega^{\ast})$ is convex block compact. Let $A$ and $B$ be bounded, closed and convex sets with distance $d(A,B) > 0$. If every $x^{\ast} \in E^{\ast}$ with
\[ \sup(x^\ast, B) < \inf(x^\ast,A) \]
attains its infimum on $A$ and its supremum on $B$, then $A$ and $B$ are both weakly compact. }

\noindent We obtain new characterizations of weakly compact sets and reflexive spaces, as well as a result concerning a variational problem in dual Banach spaces.
\end{abstract}

\section{Introduction}
\label{intro}

The well-known James' theorem \cite{jam} claims that {\it a bounded closed convex set $C$ in a Banach space $E$ is weakly compact if and only if every $x^{\ast} \in E^{\ast}$ attains its maximum on $C$}. There are many reasons why this theorem has attracted the attention of so many researchers: the great number of applications, the search for simpler proofs (see \cite{god}, \cite{kal1}, \cite{moo}, \cite{morillon}) and the concern of strengthening it (see \cite{quantitative_James}, \cite{Pfitzner_Boundary_Problem}); in this sense, there are results (see \cite{debs_godefroy-saintRaymond},\cite{noteNomrAttaining}) showing that in order to prove that $B_{E}$ is weakly compact (i.e. $E$ reflexive) we do not have to check that every functional of $E^{\ast}$ attains its maximum on $B_{E}$, but only those functionals belonging to a certain set which is ``big enough'' in a topological sense (for instance, a relative \ws-open subset of the unit sphere of $E^{\ast}$).

The aim of this paper is to find sets of functionals which determine the weak compactness of sets in the same spirit of James' compactness theorem, sets determined by geometrical properties. Our motivation is the following question raised by Delbaen: 
\begin{que}
\label{que:Delbaen}
Let $(\Omega,{\mathcal F},\PP)$ be a probability space and let $A$ be a bounded convex and closed subset of $\LL^1(\Omega,{\mathcal F},\PP)$ with $0 \notin A$. Assume that for every $Y\in \LL^\infty(\Omega,{\mathcal F}, \PP)$  with
$$ 0 < \inf \{ \E [X\cdot Y]\colon X\in A \}$$
we have that this infimum is attained. Is $A$ necessarily uniformly integrable?
\end{que}

The main result of this paper, which gives a positive answer to the previous problem, offers a new characterization of weak
compactness in Banach spaces $E$ whose dual unit ball $B_{E^{\ast}}$ endowed with the \ws-topology is convex block compact, i.e., every sequence $(x_{n}^{\ast})_{n \in \N}$ in $B_{E^{\ast}}$ has a convex block subsequence which is \ws-convergent.
\begin{thm}
\label{TEOR:MainOneSide}
Let $E$ be a Banach space such that $(B_{E^{\ast}}, \omega^{\ast})$ is convex block compact. Let $A$ and $B$ be bounded, closed and convex sets with distance $d(A,B) > 0$. If every $x^{\ast} \in E^{\ast}$ with
\[ \sup(x^\ast, B) < \inf(x^\ast,A) \]
attains its infimum on $A$ and its supremum on $B$, then $A$ and $B$ are both weakly compact. 
\end{thm}
Here, $d(A,B):= \inf{\{ \| a - b\|: a \in A, b \in B \}}$ is the usual distance between sets.
\begin{figure}[h!]
  \centering
    \includegraphics[width=2.2in]{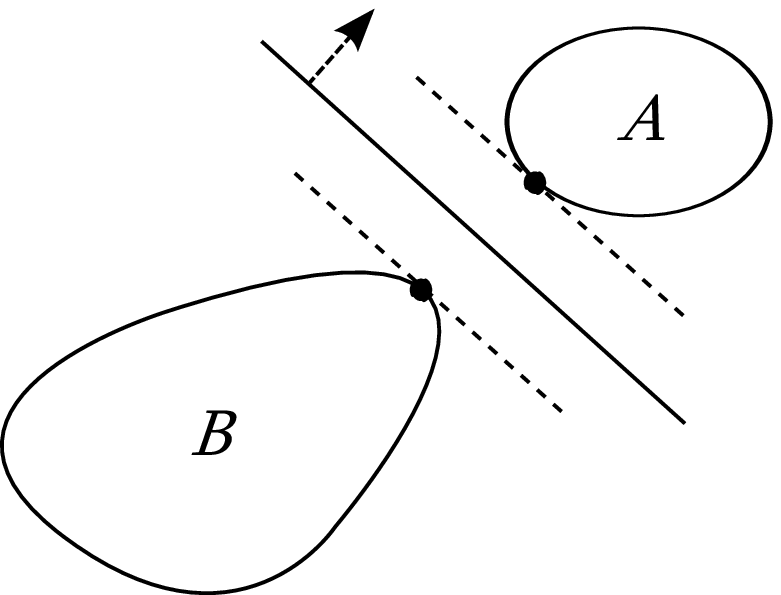}
  \label{fig:OneSideJames}
\end{figure}

As far as we know, Delbaen's problem was motivated by some questions in the framework of financial mathematics. We hope that some other applications might follow from this last theorem.

The geometrical feature of these functionals (see Figure~\ref{fig:OneSideJames}) motivates the denomination ``one-side'' to refer to this new version of James' theorem. This philosophy can be applied to obtain one-side versions of other classical results such as Bishop-Phelps theorem, for instance if we consider only support points corresponding to supporting functionals satisfying a similar geometrical property. This is discussed in section 2. 

In section 3 we revisit the concept of (I)-generation introduced by Fonf and Lindenstrauss \cite[p.159, Definition 2.1]{fon-lin}, and give in Theorem \ref{theo:UnboundedIEnvelopes} a one-side version of their result about the (I)-generation of a set by a James boundary of it. Recall that if $B \subset C$ are subsets of a dual Banach space $E^{\ast}$ then $B$ is said to be a \emph{James boundary} of $C$ if for every $x \in E$ we have that the supremum of $x$ on $C$ is attained at some point of $B$.  In our setting, $B$ has this property only for some functionals $x \in E$, which has the advantage that let us consider even unbounded sets. We show in Proposition \ref{PROP:equivGenerationSimons} the equivalence of the one-side (I)-generation formula given in Theorem \ref{theo:UnboundedIEnvelopes} and Simons' inequality. 

Section 4 is devoted to the proof of Theorem \ref{TEOR:MainOneSide}. Among the ingredients, we will need a one-side version of the classical Rainwater-Simons' theorem \cite[p.139, Theorem 3.134]{fab-haj-mont-ziz}, see Theorem \ref{TEOR:RainwaterGeneral}, . The theory developped here let us offer new characterizations of weakly compact sets and reflexive spaces, see Theorem \ref{TEOR:basic} and Corollary \ref{CORO:reflexiveCharacterization}. 

In Section 5, we prove a one-side version of Godefroy's theorem \cite[p. 174, Theorem I.2]{god} which is valid for unbounded sets, see Theorem \ref{TEOR:god}. Afterwards, we apply this result to the study of the \ws-compactness of the level-sets of convex norm lower semicontinuous maps defined on duals of certain Banach spaces, see Theorem \ref{dualsr}.

\subsection{Notation and terminology}

Most of our notation
and terminology are standard and can be
found in our standard references for Banach
spaces~\cite{fab-haj-mont-ziz}.  

Unless otherwise stated, $E$ will denote a Banach space with the norm $\NormOnly$. Given a subset~$S$
of a vector space, we write $\co{(S)}$, $\aco{(S)}$ and $\lin{(S)}$ to denote, respectively, the convex, absolutely convex and the linear hull of~$S$. If $(E,\NormOnly)$ is a normed space then $E^*$ denotes its topological dual. If $S$ is a subset of $E^*$, then $\sigma(E,S)$ denotes the topology of pointwise
convergence on $S$. Dually, if $S$ is a subset of $E$, then
$\sigma(E^*,S)$ is the topology for $E^*$ of pointwise convergence
on $S$. In particular $\sigma(E,E^*)$ and $\sigma(E^*,E)$ are the
weak ($\omega$)\index{weak topology}\index{topology!weak} and weak$^*$ ($\omega^{\ast}$)\index{weak$^*$ topology} topologies respectively.  

Given $x^{*}\in E^{*}$ and $x\in
E$, we write $\langle x^{*}, x \rangle = \langle x, x^{*} \rangle = x^{*}(x)$ for the
evaluation of~$x^{*}$ at~$x$. If $x\in E$ and $\delta >0$ we denote by $B(x,\delta)$   (or $B[x,\delta]$) the open (resp. closed) ball centered at $x$ of radius $\delta$. To simplify, we will simply write $B_{E}:=B[0,1]$; and the unit sphere \mbox{$\{x\in E\colon \|x\|=1\}$} will be denoted by $S_E$. An element $x^{\ast} \in E^{\ast}$ is \emph{norm-attaining} if there is $x \in B_{E}$ with $x^{\ast}(x) = \| x^{\ast}\|$. The set of norm-attaining functionals of $E$ is normally denoted by $NA(E)$.

Let $E,F$ be (real) vector spaces and $D \subset F$. We will denote the cone generated by $D$ as
\begin{equation} 
\label{EQUA:ConeGeneratedByD}
 \Lambda_{D} := \{ 0\} \cup \{ \lambda f\colon \lambda > 0, f \in D \}.
\end{equation}
If moreover $(E,F)$ is a dual pair (on $\R$) we define  $L_{\emptyset}(E,F) := E \setminus \{ 0\}$ and
\begin{equation} 
\label{EQUA:SetOfDirections}
L_{D}(E,F):= \{ e \in E\colon \mbox{$\sup_{f \in D}{\langle e, f \rangle} < 0$} \}  \hspace{1mm} \mbox{ if $D \neq \emptyset$}.
\end{equation}
When no confusion arises, we will simplify the notatiion and write $L_{D}$ instead of $L_{D}(E,F)$.

If $(x_{n})_{n \in \N}$ is a sequence in a vector space $E$, then then we say that $(y_{n})_{n \in \N}$ is a \emph{normalized block subsequence} of $(x_{n})_{n \in \N}$ if there exist a sequence $(A_{n})_{n \in \N}$ of disjoint finite subsets of $\mathbb{N}$ and a sequence of real numbers $(\lambda_{k})_{k \in \N}$ such that $g_{n} = \sum_{k \in A_{n}}{\lambda_{k} f_{k}}$ and $\sum_{k \in A_{n}}{|\lambda_{k}|} = 1$ for each $n \in \N$. If moreover $\lambda_{k} \geq 0$ for each $k \in \N$ then $(y_{n})_{n \in \N}$ is called a \emph{convex block subsequence}. 

We say that a subset $C$ of a topological vector space $(E, \tau)$ is \emph{block compact} (resp. \emph{convex block compact}) if every sequence of elements of $C$ admits a normalized block subsequence (resp. convex block subsequence) which $\tau$-converges to an element of $C$.

If $f\in \R^X$ ($X$ a non-empty set) we write
\[ \sup{(f,X)}:=\sup{ \{ f(x) \colon x \in X\} } \in (-\infty,+\infty],\]
\[ \inf{(f,X)}:=\inf{ \{ f(x) \colon x \in X\} } \in [-\infty,+\infty).\]
Let $f: E \rightarrow \R \cup \{ + \infty\}$ be a proper, convex and lower-semicontinuous function. Given $x_{0} \in \dom(f)$ and $\varepsilon \geq 0$ we put
\[ \partial_{\varepsilon}{f(x_{0})} := \{ x^{\ast} \in E^{\ast}\colon \langle x^{\ast}, x - x_{0} \rangle \leq f(x) - f(x_{0}) + \varepsilon \mbox{\: for each $x \in E$ } \}.\]
If $\varepsilon > 0$ then $\partial_{\varepsilon}{f(x_{0})}$ is called the $\varepsilon$-subdifferential of $f$ at $x_{0}$, while $\partial{f(x_{0})}:=\partial_{0}{f(x_{0})}$ is the subdifferential of $f$ at $x_{0}$.

\section{One-side Bishop-Phelps' theorem}
\label{sec:1}

Given $C \subset E^{\ast}$, we say that an element $c^{\ast} \in C$ is a \emph{\ws-support point} of $C$ if there exists $x \in E \setminus \{ 0\}$ with $\langle x,c^{\ast} \rangle = \sup{(x,C)}$.  One of the main features of these points is given by the next theorem.
\begin{thm}\cite[p. 177, Theorem 1]{phe3}
\label{TEOR:originalBishopPhelps}
Let $C$ be a \ws-closed convex subset of $E^{\ast}$. The set of \ws-support points of $C$ is norm-dense in the norm-boundary of $C$. 
\end{thm}

We provide now a one-side extension of the previous theorem. 
\begin{pro}
\label{PROP:OneSideBishopPhelps}
Let $C$ and $D$ be \ws-closed convex subsets of $E^{\ast}$ such that $D$ is bounded and there exists $y \in L_{D}$ with $\sup{(y,C)} < + \infty$. Then, the set
\[ \Sigma_{C}^{D}:= \{ c^{\ast} \in C \colon \langle x, c^{\ast} \rangle = \sup{(x,C)} \mbox{ \: for some }x \in L_{D}  \} \]
satisfies that $C \subseteq \overline{ \Sigma_{C}^{D} + \Lambda_{D} }^{\| \cdot \|}$.
\end{pro} 

\begin{proof}
Fix $x_{0}^{\ast} \in C$. We claim that we can write $x_{0}^{\ast} = x_{1}^{\ast} + \lambda_{1} d_{1}^{\ast}$ for some elements $\lambda_{1} \geq 0$, $d_{1}^{\ast} \in D$ and $x_{1}^{\ast} \in C$ with $(x_{1}^{\ast} - \varepsilon D) \cap C = \emptyset$ for every $\varepsilon > 0$. To see this, consider the function $f: (D, \omega^{\ast}) \rightarrow [0, +\infty)$ given by 
\[ f(d^{\ast}) = \sup{\{ \lambda \geq 0\colon x_{0}^{\ast} - \lambda d^{\ast} \in C \}}.\]
It is a bounded function, since for the element $y \in L_{D}$ of the theorem we have that $\langle y, x_{0}^{\ast} \rangle - \lambda \langle y, d^{\ast} \rangle \leq \sup{(y, C)} < +\infty \mbox{ whenever }x_{0}^{\ast} - \lambda d^{\ast} \in C$, and hence
\[ f(d^{\ast}) \leq \left( \sup{(y, C)} - \langle y, x_{0}^{\ast} \rangle \right)/\left(-\sup{(y,D)}\right) \mbox{ \: for every $d^{\ast} \in D$}. \]
Moreover $f$ is upper-semicontinuous (for the relative \ws-topology on $D$) since if $d^{\ast} \in D$ and $\alpha \in \R$ satisfy that $f(d^{\ast}) < \alpha$, then $x_{0}^{\ast} - \alpha d^{\ast} \notin C$ and using the separation theorem we can find a relatively \ws-open set $W$ in $D$ with $d^{\ast} \in W$ and $\sup{(f, W)} < \alpha$. By \ws-compactness, $f$ attains its supremum $\lambda_{1}$ at some point $d_{1}^{\ast} \in D$. It is clear that $x_{1}^{\ast}:= x_{0}^{\ast} - \lambda_{1} d_{1}^{\ast}$ together with the elements $d_{1}^{\ast}$ and $\lambda_{1}$ satisfy the claim. 

It remains to prove that $x_{1}^{\ast}$ can be approximated (in norm) by elements of $\Sigma_{C}^{D}$, which will finish the proof. Put $M:= \sup{\{ \| d^{\ast}\|\colon d^{\ast} \in D \}}$ and fix $\varepsilon > 0$. Since $(x_{1}^{\ast} - \varepsilon D) \cap C = \emptyset$, there exist $x_{1} \in E$ with
\[ \sup{(x_{1}, C)} < \inf{(x_{1}, x_{1}^{\ast} - \varepsilon D)} = \langle x_{1}, x_{1}^{\ast} \rangle - \varepsilon \sup{(x_{1}, D)}.\]
Then $\gamma := \sup{(x_{1}, D)} < 0$ and
\begin{equation} 
\label{EQUA:auxBP2}
 \sup{(x_{1}, C)} < \langle x_{1}, x_{1}^{\ast} \rangle - \varepsilon \gamma. 
\end{equation}

Consider now the function $\sigma: E \rightarrow \R$ given by $\sigma(x) = \sup{(x,C)}$. It is proper ($y \in \dom(f)$), convex and lower-semicontinuous (it is the supremum of a family of continuous funcionals on $E$). Moreover condition \eqref{EQUA:auxBP2} implies that
\[ \langle x,x_{1}^{\ast} \rangle + (\sigma(x_{1}) - \langle x_{1}, x_{1}^{\ast} \rangle) \leq \sigma(x) + (-\gamma) \varepsilon \]
or equivalently, $x_{1}^{\ast} \in \partial_{- \gamma\varepsilon}{\sigma(x_{1})}$. Using Br{\o}ndsted-Rockafellar's theorem \cite[p. 48, Theorem 3.17]{phe} we can find $z \in \dom(\sigma)$ and $z^{\ast} \in E^{\ast}$ such that
\[ z^{\ast} \in \partial{\sigma(z)}, \: \| z - x_{1}\| \leq -\gamma/(2M) \: \mbox{ and }\: \| z^{\ast} - x_{1}^{\ast}\| \leq 2 M \varepsilon. \]
To finish the proof we will check that $z^{\ast} \in \Sigma_{C}^{D}$. Condition $z^{\ast} \in \partial{\sigma(z)}$ means that
\[ \langle x, z^{\ast} \rangle + \left(\sigma(z) - \langle z, z^{\ast} \rangle \right) \leq \sigma(x) \mbox{ \: for every \: $x \in E$}. \]
Raplacing $x = \lambda y$ for arbitrary $y \in E$ and $\lambda > 0$, and making $\lambda$ tends to infinity, we get that $\langle z^{\ast}, y \rangle \leq \sigma(y)$ for every $y \in E$, so $z^{\ast} \in C$. On the other hand, taking $x=0$ we get that $\sigma(z) = \langle z^{\ast}, z \rangle$, where $z \in L_{D}$ since
\[ \sup(z,D) \leq \sup(x_{1}, D) + \frac{- \gamma}{2 M} M = \frac{\gamma}{2} < 0. \]
\end{proof}

We can particularize the previous proposition to $C=B_{E^{\ast}}$, and taking into account that the \ws-support points of $B_{E^{\ast}}$ are the norm-attaining functionals we deduce the following one-side extension of the classical Bishop-Phelps theorem.
\begin{cor}
The set 
\[ NA(E^{\ast},D) := \{ x^{\ast} \in S_{E^\ast}\colon \langle x,x^{\ast} \rangle = 1 \mbox{ \: for some $x \in L_{D} \cap S_{E}$} \}
\]
satisfies  $B_{E^{\ast}} \subseteq \overline{NA(E^{\ast},D) + \Lambda_{D}}^{\| \cdot \|}$.
\end{cor}


\section{One-side (I)-generation}
\label{sec:2}

Let $K$ be a \ws-compact convex subset of $E^{\ast}$ and $B \subset K$. In \cite[p.159, Definition 2.1]{fon-lin} Fonf and Lindenstrauss introduced the following notion: \emph{we say that $B$ (I)-generates $K$ if for every countable union of \ws-compact convex sets $\bigcup_{n=1}^{\infty}{K_{n}}$ containing $B$ we have that $K$ is contained in the norm-closed convex hull of $\bigcup_{n=1}^{\infty}{K_{n}}$}. It is clear that in this case $K$ coincides with the \ws-closed convex hull of $B$. A sufficient condition for $B$ to (I)-generate $K$ is that $B$ is a James boundary of $K$, (i.e. every $x \in E$ attains its maximum on $K$ at some point of $B$) as it is shown in \cite[p.160, Theorem 2.3]{fon-lin}. This fact was exploited afterwards by Kalenda \cite{kal1} and Moors \cite{moo} to give simple proofs of James' theorem in nonseparable cases. 

The following theorem provides a one-side extension of \cite[p.160, Theorem 2.3]{fon-lin} when $B$ satisfies the definition of James boundary only for a certain set of elements $x \in E$. We point out that, in contrast with the Fonf-Lindenstrauss result, unbounded sets $B$ are allowed here. The proof is inspired on the one of \cite[p. 99, Theorem 2]{moo}, although we avoid to use Krein-Milmann and Milmann theorems.

\begin{thm}\label{theo:UnboundedIEnvelopes}
Let $D \subset E^{\ast}$ be a \ws-compact convex set with $0 \notin D$ and let $B \subset E^{\ast}$ be a set with the following property:
\begin{enumerate}
\item[(i)] For every $x \in L_{D}$ there is $b^{\ast} \in B$ with $\langle x, b^{\ast} \rangle = \sup(x,B)$.
\end{enumerate}
If $B \subset \cup_{n=1}^{\infty}{K_{n}}$ for some family of \ws-compact convex subsets of $\overline{\co{(B)}}^{\omega^{\ast}}$ then
\[ \overline{\co{(B)}}^{\omega^{\ast}} \subset \overline{\co{(\cup_{n=1}^{\infty}{K_{n}})}  + \Lambda_{D}}^{\| \cdot \|}. \]
\end{thm}

\begin{proof}
Fix $\varepsilon > 0$. For every $n \in \N$ define $C_{n} := K_{n} + (\varepsilon/n) B_{E^{\ast}}$. It is clear that these are \ws-compact convex sets satisfying that
\[ \overline{\co{(B)}}^{\omega^{\ast}} \subseteq \Gamma := \overline{\co{(\cup_{n=1}^{\infty}{C_{n}})}}^{\omega^{\ast}}. \]
\textbf{Claim:} Given $x_{0} \in L_{D}$ there is $N \in \N$ such that 
\[ F_{x_{0}}:= \{ y^{\ast} \in \Gamma\colon \langle  x_{0}, y^{\ast} \rangle = \sup{(x_{0}, \Gamma)}\} \subset \co{\mbox{$(\cup_{n=1}^{N}{C_{n}})$}}.\] 
We can assume that $\|x_{0}\|=1$. Set $\alpha:= \sup{(x_{0}, \Gamma)} - \sup{(x_{0}, B)}$. Notice that $\alpha > 0$, because otherwise both supremums coincide and by the hypothesis there would be $b_{0}^{\ast} \in B$ with $\langle x_{0}, b_{0}^{\ast} \rangle = \sup{(x_{0}, B)} = \sup{(x_{0}, \Gamma)}$. This is not possible since $b_{0}^{\ast} + (\varepsilon/j)B_{E^{\ast}} \subseteq \Gamma$ for some $j \in \N$. Take $N \in \N$ such that $\varepsilon/N < \alpha$. This yields that
\begin{equation}
\label{EQUA:strictSupremum}
\begin{split}
\sup{(x_{0}, \mbox{$\bigcup_{n > N}{C_{n}}$})} & \leq \sup{(x_{0}, \mbox{$\bigcup_{n>N}{K_{n}}$})} + \frac{\varepsilon}{N}\\
 & < \sup{(x_{0}, \overline{\co{(B)}}^{\omega^{\ast}})} + \alpha
= \sup{(x_{0}, \Gamma)}.
\end{split}
\end{equation}
Notice that
\[ \Gamma=\overline{\co{(\mbox{$\bigcup_{n=1}^{\infty}{C_{n}}$})}}^{\omega^{\ast}} = \co{\left( \co{(\mbox{$\bigcup_{n=1}^{N}{C_{n}}$})} \cup \overline{\co{(\mbox{$\bigcup_{n > N}{C_{n}}$})}}^{\omega^{\ast}} \right)}. \]
Therefore, given $y^{\ast} \in F_{x_{0}}$ we can write it as $y^{\ast} = \lambda x^{\ast}_{1} + (1 - \lambda)x^{\ast}_{2}$ for some $\lambda \in [0,1]$, $x_{1}^{\ast} \in \co{(\bigcup_{n=1}^{N}{C_{n}})}$ and $x_{2}^{\ast} \in \overline{\co{(\bigcup_{n > N}{C_{n}})}}^{\omega^{\ast}}$. But \eqref{EQUA:strictSupremum} gives
\[ \langle x_{0}, x_{2}^{\ast} \rangle \leq \sup{(x_{0}, \mbox{$\bigcup_{n > N}{C_{n}}$})} < \sup{(x_{0}, \Gamma)}, \] 
so if $\lambda < 1$ then
\[ \langle x_{0}, y^{\ast} \rangle = \lambda \langle x_{0}, x_{1}^{\ast} \rangle + (1 - \lambda) \langle x_{0}, x_{2}^{\ast} \rangle < \sup{(x_{0}, \Gamma)}. \]
which contradicts that $y^{\ast} \in F_{x_{0}}$. Thus $\lambda = 1$ and we conclude that $y^{\ast} = x_{1}^{\ast}$ belongs to $\co{(\bigcup_{n=1}^{N}{C_{n}})}$. This proves the claim.\\

\noindent We can now finish the proof of the theorem. By the claim we have that
\[ \Sigma_{\Gamma}^{D} = \bigcup{\{F_{x_{0}}\colon x_{0} \in L_{D}\}} \subseteq \co{(\mbox{$\bigcup_{n=1}^{\infty}{C_{n}}$})} \subseteq \co{(\mbox{$\bigcup_{n=1}^{\infty}{K_{n}}$})} + \varepsilon B_{E^{\ast}}. \]
We now distinguish two cases: If $D \neq \emptyset$, then using Proposition \ref{PROP:OneSideBishopPhelps} we deduce that
\[
 \Gamma \subseteq \overline{ \Sigma_{\Gamma}^{D} + \Lambda_{D}}^{\| \cdot \|} \subseteq \co{(\mbox{$\bigcup_{n=1}^{\infty}{K_{n}}$})} + \Lambda_{D} + 2 \varepsilon B_{E^{\ast}}. 
 \]
On the other hand, if $D= \emptyset$ then $L_{D} = E \setminus \{ 0\}$ and $\Lambda_{D} = \{ 0\}$, so $\Sigma_{\Gamma}^{D}$ is the set of all \ws-support points of $\Gamma$. By Theorem \ref{TEOR:originalBishopPhelps} we obtain that its norm-boundary satisfies 
 \[ \partial{\Gamma} = \overline{ \Sigma_{\Gamma}^{D}}^{\| \cdot \|} = \overline{ \Sigma_{\Gamma}^{D} + \Lambda_{D}}^{\| \cdot \|} \subseteq \co{(\mbox{$\bigcup_{n=1}^{\infty}{K_{n}}$})} + \Lambda_{D} + 2 \varepsilon B_{E^{\ast}}  \]
 But $\Gamma$ is bounded by condition (i), so it is clear that $\co{(\partial{\Gamma})} = \Gamma$ and we obtain in both cases that
\[ \Gamma \subseteq \co{(\mbox{$\bigcup_{n=1}^{\infty}{K_{n}}$})} + \Lambda_{D} + 2 \varepsilon B_{E^{\ast}} \]
Since $\overline{\co{(B)}}^{\omega^{\ast}} \subseteq \Gamma$ and $\varepsilon > 0$ was arbitrary the proof is finished.
\end{proof}

In \cite[Theorem 2.2]{cas-alt16} it is shown that the property of (I)-generation is equivalent to Simons' $\sup - \limsup$ theorem (fact that appears implicitly in \cite[Lemma 2.1 and Remark 2.2]{kal1}) and also to Simon's inequality. A similar equivalence can be also established in our one-side case. We put
\[ \co_{\sigma_{p}}{\{ x_{n}\colon n \geq 1 \}} := \left\{ \sum_{n=1}^{+ \infty}{\lambda_{n} x_{n}}\colon \mbox{ for all $n \geq 1$, $\lambda_{n} \geq 0$ and $\sum_{n=1}^{\infty}{\lambda_{n}} = 1$} \right\}. \]
The following Proposition answers \cite[Question 10.9]{Compactness-optimality-and-risk}.
\begin{pro}
\label{PROP:equivGenerationSimons}
Let $D \subset E^{\ast}$ be a \ws-compact convex set with $0 \notin D$ and let $C \subset E^{\ast}$ be a \ws-closed convex set such that every $x \in L_{D}$ has finite supremum on $C$. Given $B \subset C$ the following assertions are equivalent:
\begin{enumerate}
\item[(i)] For every covering $B \subseteq \bigcup_{n =1}^{\infty}{K_{n}}$ by an increasing sequence of \ws-compact convex subsets $K_{n} \subset C$ we have that
\[ C \subset \overline{\mbox{$\bigcup_{n=1}^{\infty}{K_{n}}$} + \Lambda_{D}}^{\| \cdot \|}. \]
\item[(ii)] For every bounded sequence $(x_{n})_{n \in \N}$ in $L_{D}$
\[ \sup_{b^{\ast} \in B}{(\limsup_{n}{\langle b^{\ast},x_{n} \rangle})} = \sup_{c^{\ast} \in C}{(\limsup_{n}{\langle c^{\ast}, x_{n} \rangle})}. \]
\item[(iii)] For every bounded sequence $(x_{n})_{n \in \N}$ in $L_{D}$ 
\[ \sup_{b^{\ast} \in B}{(\limsup_{n}{\langle b^{\ast},x_{n} \rangle})} \geq \inf_{x \in \co_{\sigma_{p}}{\{ x_{n}\colon n \geq 1 \}}}{ \sup_{c^{\ast} \in C}{\langle c^{\ast}, x}\rangle }. \]
\end{enumerate}
\end{pro}

\begin{proof}
(i) $\Rightarrow$ (ii): Fix a bounded sequence $(x_{n})_{n \in \N}$ in $L_{D}$ and write
\[ \alpha =\sup_{b^{\ast} \in B}{(\limsup_{n}{\langle b^{\ast},x_{n} \rangle})}. \]
If $\alpha = + \infty$ then the equality is clear. Otherwise, fix $\varepsilon > 0$ and define the sequence of \ws-compact convex sets
\[ K_{n} := \{ x^{\ast} \in C \cap n B_{E^{\ast}}\colon \langle x^{\ast},x_{k} \rangle \leq \alpha + \varepsilon \mbox{ for every $k \geq n$} \}. \] 
Since $B \subset \bigcup_{n=1}^{\infty}{K_{n}}$, it follows from (i) that $ C \subset \overline{\bigcup_{n=1}^{\infty}{K_{n}} + \Lambda_{D}}^{\| \cdot \|}$. But for every $x^{\ast} \in \bigcup_{n=1}^{\infty}{K_{n}}, \lambda \geq 0, d^{\ast} \in D$ we have that
\[ \limsup_{n}{\langle x^{\ast}+ \lambda d^{\ast}, x_{n} \rangle} \leq \limsup_{n}{\langle x^{\ast}, x_{n} \rangle} \leq \alpha + \varepsilon, \]
 so \mbox{$\limsup_{n}{\langle c^{\ast}, x_{n} \rangle} \leq \alpha + \varepsilon$} for each $c^{\ast} \in C$.

(ii) $\Rightarrow$ (iii): It will be enough to show that
\begin{equation} 
\label{EQUA:inequalitySimons}
\sup_{c^{\ast} \in C}{(\limsup_{n}{\langle c^{\ast},x_{n} \rangle})} \geq \inf_{x \in \co_{\sigma_{p}}{\{ x_{n}\colon n \geq 1 \}}}{ \sup_{c^{\ast} \in C}{\langle c^{\ast}, x}\rangle }. 
\end{equation}
This will be a consequence of the unbounded version of Simons' inequality (see \cite[Theorem 10.5]{Compactness-optimality-and-risk}) once we have checked that every $x \in \co_{\sigma_{p}}{\{ x_{n}\colon n \geq 1 \}}$ attains its supremum on $C$. Fix such an $x$ and notice that $x \in L_{D}$. Assume that $x$ does not attain its (finite) supremum $\alpha$ on $C$, so there is a sequence of elements $(x_{n}^{\ast})_{n \in \N}$ in $C$ satisfying
\[ \alpha - 1 \leq \langle x_{0}, x^{\ast}_{n} \rangle \leq \alpha \mbox{ \hspace{2mm} and \hspace{2mm} } \| x_{n}^{\ast}\| \geq n \mbox{ for each $n \in \N$. } \] 
By the Uniform Boundedness Principle there exists $z \in S_{E}$ such that $(\langle z, x_{n}^{\ast} \rangle)_{n \in \N}$ is unbounded above. Since $L_{D}$ is an open set, we have that $x_{0} + \varepsilon z \in L_{D}$ for $\varepsilon > 0$ small enough. But this element has not bounded supremum on $C$, which contradicts the hypothesis of the theorem.

(iii) $\Rightarrow$ (i): We will proceed by contradiction and suppose that there are $z_{0}^{\ast} \in C$ and $\delta > 0$ such that $z_{0}^{\ast} + \delta B_{E^{\ast}}$ has empty intersection with $\bigcup_{n = 1}^{\infty}{K_{n}} + \Lambda_{D}$. For each $n \in \N$ the set $K_{n} + \Lambda_{D}$ is \ws-closed, so we can find $x_{n} \in S_{E}$ and $\alpha_{n} \in \R$ satisfying
\[ \langle x_{n}, z_{0}^{\ast} \rangle > \alpha_{n} + \delta > \alpha_{n} > \langle x_{n}, x^{\ast} + \lambda d^{\ast} \rangle \mbox{ \: for every \: $x^{\ast} \in K_{n}, \lambda \geq 0, d^{\ast} \in D$}. \]
In particular, this implies that $\langle x_{n}, d^{\ast} \rangle \leq 0$ for every $d^{\ast} \in D$. Fix $y_{0} \in L_{D} \cap S_{E}$ and write \mbox{$\gamma:= \sup_{x^{\ast} \in C}{\langle y_{0}, x^{\ast} \rangle} < +\infty$}. Define for each $n \in \N$ the element 
\[ y_{n}:= x_{n} + \frac{\delta}{3(1 + |\gamma| + |\langle y_{0},z_{0}^{\ast} \rangle|)}y_{0}. \]
It satisfies that $y_{n} \in L_{D}$. Moreover
\begin{equation}
\label{EQUA:boundz} 
\langle y_{n}, z_{0}^{\ast}\rangle = \langle x_{n}, z_{0}^{\ast}\rangle + \langle y_{0}, z_{0}^{\ast} \rangle \frac{\delta}{3(1 + |\gamma| + |\langle y_{0}, z_{0}^{\ast} \rangle|)} \geq  \alpha_{n} + \frac{2 \delta}{3} 
\end{equation}
and whenever $b^{\ast} \in B \cap K_{n}$ we have that
\begin{equation} 
\label{EQUA:boundB}
\langle y_{n}, b^{\ast} \rangle = \langle x_{n}, b^{\ast} \rangle + \frac{\delta}{3(1 + |\gamma| + |\langle y_{0}, z_{0}^{\ast} \rangle|)}\langle y_{0}, b^{\ast} \rangle < \alpha_{n} + \frac{\delta}{3}.
\end{equation}
Fix now a \ws-cluster point $y^{\ast \ast} \in E^{\ast \ast}$ of $(y_{n})_{n \in \N}$. By taking a subsequence we can assume that $\lim_{n}{\langle y_{n}, z_{0}^{\ast}\rangle} = \langle y^{\ast \ast}, z_{0}^{\ast} \rangle$ and moreover
\begin{equation}
\label{EQUA:aux2Equiv}
\langle y_{n}, z_{0}^{\ast} \rangle > \langle y^{\ast \ast}, z_{0}^{\ast} \rangle - \frac{\delta}{6} \mbox{ for each $n \in \N$}.
\end{equation} 
Given $b^{\ast} \in B$ there is $n_{0}\in \N$ such that $b^{\ast} \in K_{n}$ for each $n \geq n_{0}$. Therefore, we can combine \eqref{EQUA:boundz} and \eqref{EQUA:boundB}, and taking superior limits on $n$ one gets that
\begin{equation}
\label{EQUA:aux3Equiv}
\langle y^{\ast \ast}, z_{0}^{\ast} \rangle - \frac{ 2 \delta}{3} \geq \limsup_{n}{\langle y_{n}, b^{\ast} \rangle} - \frac{\delta}{3} \mbox{ for every $b^{\ast} \in B$}.
\end{equation}
Finally
\[
\begin{split}
 \langle y^{\ast \ast}, z_{0}^{\ast} \rangle - \frac{2 \delta}{3} & \stackrel{(\ref{EQUA:aux3Equiv})}{\geq} \sup_{b^{\ast} \in B}{\limsup_{n}{\langle y_{n}, b^{\ast} \rangle}} - \frac{\delta}{3} \stackrel{(iii)}{\geq} \inf_{y \in \co_{\sigma_{p}}{\{ y_{n}\colon n \geq 1 \}}}{\sup_{x^{\ast} \in C}{\langle y, x^{\ast} \rangle}} - \frac{\delta}{3}\\ 
& \geq  \inf_{y \in \co_{\sigma_{p}}{\{ y_{n}\colon n \geq 1 \}}}{\langle y, z_{0}^{\ast} \rangle} - \frac{\delta}{3} \stackrel{(\ref{EQUA:aux2Equiv})}{\geq} \langle y^{\ast \ast}, z_{0}^{\ast} \rangle - \frac{\delta}{2}
\end{split} 
\]
which leads to $0 \geq \delta/6$. This contradiction finishes the proof. 
\end{proof}

Remark that this last proposition together with the unbounded version of Simons' inequality \cite[Theorem 10.5]{Compactness-optimality-and-risk} gives an alternative proof of Theorem \ref{theo:UnboundedIEnvelopes}.


\section{One-side James' theorem}
\label{sec:3}

The aim of this section is to prove the main Theorem \ref{TEOR:MainOneSide}. The first auxiliary result that we will need is the following theorem which is a one-side version of the classical Rainwater-Simons' theorem \cite[p.139, Theorem 3.134]{fab-haj-mont-ziz}. We also refer to \cite[p. 100, Corollary 3]{moo} for a proof of Rainwater-Simons' theorem using the property of (I)-generation. Indeed, we will base on this last approach but with the one-side version of the (I)-generation that was developed in Theorem \ref{theo:UnboundedIEnvelopes}.

\begin{thm}
\label{TEOR:RainwaterGeneral}
Let $B$ be a bounded subset of $E^{\ast}$ with $0 \notin K := \overline{\co{(B)}}^{\omega^{\ast}}$ and satisfying:
\begin{enumerate}
\item[(i)] For every $x \in L_{K}$ there is $b^{\ast} \in B$ with $\langle x, b^{\ast}\rangle = \sup(x,B) = \sup(x,K)$.
\end{enumerate}
If $(x_{n})_{n \in \N}$ is a bounded sequence in $E$ which $\sigma(E,B)$-converges to $0$, then it is $\sigma(E,K)$-convergent to $0$.
\end{thm}

\begin{proof}
We can assume that the sequence $(x_{n})_{n \in \N}$ is contained in $B_{E}$. We will fix $z \in L_{K} \cap S_{E}$ and write $\gamma = \sup(z,K) < 0$.

Take $\varepsilon > 0$ and define for each $n \in \N$ 
\[ K_{n}^{\varepsilon} := \{ y^{\ast} \in K \cap n B_{E^{\ast}}\colon | \langle x_{k},y^{\ast} \rangle| \leq \varepsilon \mbox{ for all $k \geq n$}\}.\] 
This is an increasing sequence of \ws-compact convex subsets of $K$ whose union contains $B$ by the hypothesis. Theorem \ref{theo:UnboundedIEnvelopes} yields that
\begin{equation} 
\label{EQUA:auxRainwater}
K \subseteq \overline{\mbox{$\bigcup_{n=1}^{\infty}{K_{n}^{\varepsilon}}$} + \Lambda_{K}}^{\| \cdot \|}. 
\end{equation}
Fix an arbitrary $x_{0}^{\ast} \in K$ and put $\lambda_{0}=1$. Using (\ref{EQUA:auxRainwater}) we can construct inductively sequences $(x_{n}^{\ast})_{n \in \N}$ in $K$, $(y_{n}^{\ast})_{n \in \N}$ in $\bigcup_{n=1}^{\infty}{K_{n}^{\varepsilon}}$ and $(\lambda_{n})_{n \in \N}$ in $[0, + \infty)$ so that for every $n \geq 1$ 
\[ \left\| x_{n-1}^{\ast} - (y_{n}^{\ast} + \lambda_{n} x_{n}^{\ast}) \right\| < \frac{\varepsilon}{2^{n-1}(1 + \lambda_{0}) \cdot \ldots \cdot (1 + \lambda_{n-1})}. \]
A simple argument by induction shows that
\begin{equation} 
\label{EQUA:aux2Rainwater}
\left\| x_{0}^{\ast} - \left(\sum_{k=1}^{n}{\lambda_{0} \cdot \ldots \cdot \lambda_{k-1} y_{k}^{\ast}}\right) - \lambda_{0}\cdot \ldots \cdot \lambda_{n} x_{n}^{\ast} \right\| < \varepsilon \sum_{k=1}^{n}{\frac{1}{2^{k}}}. 
\end{equation}
On the other hand, equation \eqref{EQUA:aux2Rainwater} yields that
\[
\begin{split}
 \langle z, x_{0}^{\ast} \rangle & \leq \varepsilon + \left(\sum_{k=1}^{n}{\lambda_{0} \cdot \ldots \cdot \lambda_{k-1} \langle z, y_{k}^{\ast} \rangle}\right) + \lambda_{0} \cdot \ldots \cdot \lambda_{n} \langle z, x_{n}^{\ast} \rangle\\ 
 & \leq \varepsilon + \gamma \sum_{k=1}^{n+1}{\lambda_{0} \cdot \ldots \cdot \lambda_{k-1}}. 
\end{split} 
 \]
This implies that
 \begin{equation} 
 \label{EQUA:convergentSeries}
 \sum_{k=1}^{\infty}{\lambda_{0} \cdot \ldots \cdot \lambda_{k-1}} \leq \frac{\varepsilon - \langle z,x_{0}^{\ast} \rangle}{- \gamma}. 
 \end{equation}
Since $K$ is a bounded set and the previous series is absolutely convergent, it follows from \eqref{EQUA:aux2Rainwater} that
\[ \left\| x_{0}^{\ast} - \sum_{k=1}^{+\infty}{\lambda_{0} \cdot \ldots \cdot \lambda_{k-1} y_{k}^{\ast}} \right\| \leq \varepsilon. \]
Therefore, for each $n \in \N$ we have that
 \[
 |\langle x_{n}, x_{0}^{\ast} \rangle| \leq \varepsilon + \sum_{k=1}^{+ \infty}{\lambda_{0} \cdot \ldots \cdot \lambda_{k-1} |\langle x_{n},y_{k}^{\ast}\rangle|}.
 \]
Taking the superior limit on $n$ we deduce that
 \[ \limsup_{n}{| \langle x_{n},x_{0}^{\ast}\rangle |} \leq \varepsilon \left( 1 + \sum_{k=1}^{+\infty}{\lambda_{0} \cdot \ldots \cdot \lambda_{k-1}} \right) \leq \varepsilon \left( 1 + \frac{\varepsilon - \langle z, x_{0}^{\ast} \rangle}{- \gamma}\right).\]
Since $\varepsilon > 0$ and $x_{0}^{\ast} \in K$ were arbitrary, we conclude the result. 
\end{proof}

\begin{thm}
\label{TEOR:GodefroyTypeTheorem2}
Let $D$ be a convex \ws-compact subset of $E^{\ast}$ with $0 \notin D$ and let $B \subset E^{\ast}$ be a set such that
\begin{itemize}
\item[(i)] For each $x \in L_{D}$ there is $b^{\ast} \in B$ with $\langle x, b^{\ast} \rangle = \sup(x,B)$.
\item[(ii)] Every bounded sequence $(x_{n})_{n \in \N}$ in $E$ contains a convex block subsequence which is $\sigma(E, B \cup D)$-convergent.
\end{itemize}  
Then 
$$\overline{\co{(B)}}^{\omega^{\ast}} = \overline{\co{(B)}}^{\| \cdot \|}.$$
\end{thm}

\begin{proof}
The set $L_{D}$ is open, and by (i) we can write it as the union $L_{D} = \bigcup_{n=1}^{\infty}{H_{n}}$ where $H_{n}$ is the closed set given by
\[ H_{n}:=\{ x \in L_{D}\colon \sup{(x, \co{(B)})} \leq n \}. \]
Baire's category theorem implies that there are $M \in \N$, $a \in L_{D}$ and $\delta > 0$ such that $a + \delta B_{E}\subset H_{M}$, which yields that
\begin{equation} 
\label{EQUA:BaireApplication}
\langle b^{\ast}, a \rangle + \delta \| b^{\ast}\| \leq M \mbox{ for every $b^{\ast} \in \co{(B)}$.}
\end{equation}
We will assume now that there exists an element $b_{0}^{\ast} \in \overline{\co{(B)}}^{\omega^{\ast}} \setminus \overline{{\rm co}(B)}^{\|\cdot\|}$ and we will get a contradiction. The separation theorem \cite[p. 59, Theorem 2.12]{fab-haj-mont-ziz} provides $x^{**}_0\in B_{E^{**}}$  and real numbers $\alpha, \beta$ satisfying 
\begin{equation}\label{EQUA:auxUnboundedGodefroy2}
\langle x^{**}_0,b_{0}^*\rangle > \beta > \alpha > \langle x^{**}_0, x^* \rangle \mbox{ for every $x^*\in \overline{{\rm co}(B)}^{\|\cdot\|}$.}
\end{equation}
We claim that we can then construct inductively sequences $(x_{m}^{\ast})_{m \in \N}$ in ${\rm co}(B)$ and $(x_{n})_{n \in \N}$ in $B_{E}$ with the following properties:
\begin{enumerate}
\item[(a)] $\langle b_{0}^{\ast}, x_{n}\rangle > \beta$ \: for each $n \in \N$,
\item[(b)] $\langle x_{m}^{\ast}, x_{n} \rangle > \beta$ \: if $m>n$,
\item[(c)] $\langle x_{m}^{\ast}, x_{n} \rangle < \alpha$ \: if  $m \leq n$,
\item[(d)] $\langle x_{m}^{\ast} , a \rangle > \langle b_{0}^{\ast}, a \rangle - 1$.
\end{enumerate}
To see this, take $x_{1}^{\ast} \in {\rm co}(B)$ with \mbox{$\langle x_{1}^{\ast} , a \rangle > \langle b_{0}^{\ast}, a \rangle - 1$} by \ws-density; and apply Goldstine's theorem \cite[p. 125, Theorem 3.96]{fab-haj-mont-ziz} to get $x_{1} \in B_{E}$ such that
\[ \langle x_{1}, b_{0}^{\ast}\rangle > \beta > \alpha > \langle x_{1}, x_{1}^{\ast} \rangle. \]
Suppose that we have constructed $(x_{m}^{\ast})_{m \leq k}$ and $(x_{n})_{n \leq k}$ satisfying the properties above. By \ws-density, it follows from (a) that there is $x_{k+1}^{\ast} \in {\rm co}(B)$ with
\[ x_{k+1}^{\ast} \in \bigcap_{n=1}^{k}{\{ x^{\ast}\colon \langle x^{\ast}, x_{n}\rangle > \beta \}} \mbox{ \hspace{3mm} and \hspace{3mm} } \langle x_{k+1}^{\ast} , a \rangle > \langle b_{0}^{\ast}, a \rangle - 1. \]
Using Goldstine's Theorem with $x_{0}^{\ast \ast}$, we can find $x_{k+1} \in B_{E}$ with
\[ \langle b^{\ast}, x_{k+1} \rangle > \beta > \alpha > \langle x_{m}^{\ast}, x_{k+1} \rangle \mbox{ \: for every \: $1 \leq m \leq k+1$}. \]
This proves the claim.

If we combine condition (d) and inequality \eqref{EQUA:BaireApplication} we deduce that $(x_{m}^{\ast})_{m \in \N}$ is bounded. Fix $x_{\infty}^{\ast}$ a \ws-cluster point of the sequence $(x_{m}^{\ast})_{m \in \N}$. By (ii), there is a convex block subsequence $(y_{n})_{n \in \N}$ of $(x_{n})_{n \in \N}$ that converges to some $y_{\infty} \in E$ on $B \cup D$. This fact together with (c) implies that
$ \langle x_{m}^{\ast}, y_{\infty} \rangle \leq \alpha$ for every $m \in \N$, and hence
\begin{equation}
\label{EQUA:auxAlpha2}
\langle x_{\infty}^{\ast}, y_{\infty} \rangle \leq \alpha.
\end{equation}
On the other hand, it follows from (b) that
\begin{equation}
\label{EQUA:auxBeta2}
\langle x_{\infty}^{\ast}, x_{n} \rangle \geq \beta \mbox{\hspace{2mm} for every $n \in \N$}.
\end{equation} 
Put $C:=\overline{\co{(B)}}^{\omega^{\ast}}$. Fix $\varepsilon > 0$ and define for each $n \in \N$ the \ws-compact convex set
\[ K_{n}^{\varepsilon} := \{ x^{\ast} \in C \cap n B_{E^{\ast}}\colon |\langle x^{\ast}, y_{\infty} - y_{k} \rangle| \leq \varepsilon \mbox{ \: for every\: $k \geq n$} \}. \]
The union of all $K_{n}^{\varepsilon}$'s contains $B$, so Theorem \ref{theo:UnboundedIEnvelopes} yields that
\[ C \subset \overline{\co{(\mbox{$\bigcup_{n=1}^{\infty}{K_{n}^{\varepsilon}}$})} + \Lambda_{D}}^{\| \cdot \|} = \overline{\mbox{$\bigcup_{n=1}^{\infty}{K_{n}^{\varepsilon}}$} + \Lambda_{D}}^{\| \cdot \|}. \]
We can then find $j \in \N$, $b_{\varepsilon}^{\ast} \in K_{j}^{\varepsilon}$, $\lambda_{\varepsilon} \geq 0$ and $d^{\ast}_{\varepsilon} \in D$ with
\[ \| x_{\infty}^{\ast} - (b_{\varepsilon}^{\ast} + \lambda_{\varepsilon} d_{\varepsilon}^{\ast}) \| < \varepsilon/(1 + \| y_{\infty}\|), \]
so that for each $k \geq j$
\[ 
\begin{split}
| \langle x_{\infty}^{\ast}, y_{\infty} - y_{k} \rangle | & \leq \| y_{\infty} - y_{k}\| \cdot \| x_{\infty}^{\ast} - (b_{\varepsilon}^{\ast} + \lambda_{\varepsilon} d_{\varepsilon}^{\ast}) \| + |\langle b_{\varepsilon}^{\ast} + \lambda_{\varepsilon} d_{\varepsilon}^{\ast}, y_{\infty} - y_{k} \rangle|\\
 & \leq \varepsilon + |\langle b_{\varepsilon}^{\ast} + \lambda_{\varepsilon} d_{\varepsilon}^{\ast}, y_{\infty} - y_{k} \rangle|.
\end{split} 
\]
Therefore
\[ \limsup_{k}{| \langle x_{\infty}^{\ast}, y_{\infty} - y_{k} \rangle |} \leq \varepsilon. \]
Since $\varepsilon > 0$ is arbitrary, we have that
\[ \langle x_{\infty}^{\ast}, y_{\infty} \rangle = \lim_{k}{\langle x_{\infty}^{\ast}, y_{k}} \rangle \geq \beta, \]
where the last inequality is consequence of (\ref{EQUA:auxBeta2}) and the definition of $(y_{n})_{n \in \N}$. But this contradicts (\ref{EQUA:auxAlpha2}).
\end{proof}

We are now ready to prove the main result of this paper.

\begin{proof}[Proof of Theorem \ref{TEOR:MainOneSide}]

We divide the proof in two parts: firstly we will deal with the particular case $A=\{0\}$, and secondly we will deduce the general case as an easy consequence.

Assume that $A= \{ 0\}$ and write $D := \overline{B}^{\sigma(E^{\ast \ast}, E^{\ast})} \subset E^{\ast \ast}$. Regarding $B$ and $D$ as subsets of $E^{\ast \ast}$, they both satisfy conditions  (i) and (ii) of Theorem \ref{TEOR:GodefroyTypeTheorem2} as we now check. If  $x^{\ast} \in L_{D}(E^{\ast}, E^{\ast \ast})$ then $\sup(x^{\ast}, B) < 0$, so it attains its supremum on $B$ by hypothesis. This shows (i). On the other hand, the hypothesis of $(B_{E^{\ast}}, \omega^{\ast})$ being convex block compact implies that, in particular, $(B_{E^{\ast}}, \sigma(E^{\ast}, B))$ is convex block compact. We can now apply Theorem \ref{TEOR:RainwaterGeneral} to deduce that $(B_{E^{\ast}}, \sigma(E^{\ast},B \cup D))$ is convex block compact. Hence, using Theorem \ref{TEOR:GodefroyTypeTheorem2} we obtain that
\[ \overline{B}^{\sigma(E^{\ast \ast}, E^{\ast})} = B. \]

To prove the general case, consider $\hat{B}:= \overline{B- A}^{\| \cdot \|}$ and $\hat{A}:=0$. Notice that every functional $x_{0}^{\ast}\in E^{\ast}$ with $\sup{(x_{0}^{\ast}, \hat{B})} < 0$ satisfies that $\sup{(x_{0}^{\ast}, B)} < \inf{(x_{0}^{\ast}, A)}$, so it attains its supremum on $B$ at some $b_{0} \in B$ and its infimum on $A$ at some $a_{0} \in A$ by the hypothesis. This clearly implies that $x_{0}^{\ast}$ attains its supremum on $\hat{B}$ at $b_{0} - a_{0}$. Therefore, $\hat{A}$ and $\hat{B}$ are under the conditions of the first part of the proof and we can conclude that $\hat{B}$ is weakly compact, from where it follows that $A$ and $B$ are both weakly compact.
\end{proof}

We will show another application of Theorem \ref{TEOR:GodefroyTypeTheorem2}.

\begin{thm}\label{TEOR:basic}
Let $E$ be a Banach space such that $(B_{E^{\ast}}, \omega^{\ast})$ is convex block compact and let $D$ be a weakly relatively compact subset of $E$ with $L_{D}(E^{\ast}, E) \neq \emptyset$. If $B$ is a bounded subset of $E$ such that every element of $L_D (E^*, E)$ attains its supremum on $B$, then $B$ is weakly relatively compact.
\end{thm}
\begin{proof}
We can assume that $D$ is convex and closed, since $\overline{\co{(D)}}^{\| \cdot \|}$ is also weakly compact in $E$ by Krein's Theorem \cite[p. 138, Theorem 3.133]{fab-haj-mont-ziz} and $L_{\overline{\co{(D)}}}(E^{\ast}, E) = L_{D}(E^{\ast}, E)$. Using the natural embedding $E \subset E^{\ast \ast}$ we can regard $B$ and $D$ as subsets of $E^{\ast \ast}$, so $D$ is \ws-compact and $L_{D}(E^{\ast}, E^{\ast \ast}) = L_{D}(E^{\ast}, E)$. Therefore, the hypothesis of Theorem \ref{TEOR:GodefroyTypeTheorem2} are satisfied, and we conclude that $\overline{\co{(B)}}^{\omega^{\ast}} = \overline{\co{(B)}}^{\|\cdot\|}$, which finishes the proof.
\end{proof}

If $(\Omega,{\mathcal F},\PP)$ is a probability space then the Banach space of all integrable functions $E = \LL^1(\Omega,{\mathcal F},\PP)$ is weakly compactly generated (see \cite[p. 576]{fab-haj-mont-ziz}), and hence the unit ball of its dual $B_{E^{\ast}}$ is weak$^{\ast}$-sequentially compact, where $E^{\ast} = \LL^\infty(\Omega,{\mathcal F},\PP)$. On the other hand, recall that a bounded subset $A$ is uniformly integrable if and only if it is weakly relatively compact by Dunford's criterion (see \cite[p. 76, Theorem 15]{die}). With this in mind, Theorem \ref{TEOR:basic} gives a positive answer to Question \ref{que:Delbaen}. We also particularize Theorem \ref{TEOR:basic} to this case in the following corollary.

\begin{cor}
Let $H$ be a uniformly integrable subset of $\LL^1(\Omega,{\mathcal F},\PP)$ with $0 \notin \overline{\co(H)}$. Suppose that $A$ is a subset of $\LL^1(\Omega,{\mathcal F},\PP)$ such that every $Y\in \LL^\infty(\Omega,{\mathcal F}, \PP)$  with $\inf \{\E [X\cdot Y]\colon X\in H \} > 0$ satisfies that $\inf{\{ \E[X \cdot Z]\colon Z \in A \}}$ is attained. Then $A$ is uniformly integrable.
\end{cor}

\begin{proof}
An element $Y\in \LL^\infty(\Omega,{\mathcal F}, \PP) = \LL^1(\Omega,{\mathcal F},\PP)^{\ast}$ satisfies
$$\inf \{\E [X\cdot Y]\colon X\in H \} > 0$$
if and only if $-Y$ belongs to  $L_{H}(\LL^{\infty}, \LL^{1})$; and it attains its infimum on $A$ if and only if $-Y$ attains its supremum on the same set. The result now follows from Theorem \ref{TEOR:basic}
\end{proof}

We finish the section with a characterization of reflexivity for some Banach spaces.

\begin{cor}
\label{CORO:reflexiveCharacterization}
Let $E$ be a Banach space with $(B_{E^{\ast}}, \omega^{\ast})$ convex block compact. If there are $\alpha \in \R$ and a convex weakly compact set $D \subset E$ such that
\[ \emptyset \neq P := \{ x^{\ast} \in S_{E^{\ast}}\colon \sup{(x^{\ast}, D)} < \alpha \} \subset NA(E) \] 
then $E$ is reflexive.
\end{cor}
\begin{proof}
We can assume that $D + \alpha B_{E} \subset B_{E^{\ast \ast}} = \overline{B_{E}}^{\omega^{\ast}}$, replacing otherwise $D$ and $\alpha$ by $\mu D$ and $\mu \alpha$ with $\mu > 0$ small enough. If $\alpha \geq 0$ then $L_{D} \subset P \subset NA(E)$, and Theorem \ref{TEOR:basic} gives that $B_{E}$ is weakly compact. If $\alpha < 0$, then $x^{\ast} \in P$ if and only if 
\[\sup{(x^{\ast}, D + \alpha B_{E})} = \sup{(x^{\ast}, D)} - \alpha < 0.\] 
Thus the set $D':=D + \alpha B_{E^{\ast \ast}} \subset B_{E^{\ast \ast}} = \overline{B_{E}}^{\omega^{\ast}}$ is bounded, convex, closed and every $x^{\ast} \in L_{D'} = P$ attains its supremum on $D'$ (since it attains its supremum on $D$ by weak compactness, but also on $B_{E}$ by the hypothesis), so Theorem \ref{TEOR:MainOneSide} yields that $B_{E}$ is weakly compact.  
\end{proof}


\section{Unbounded nonlinear weak$^\ast$-James' theorem} 
\label{sec:4}

We are going to present now an unbounded version of Godefroy's theorem \cite[p. 174, Theorem I.2]{god}. A norm separable James boundary $B$ of a convex, bounded and closed subset $C$ of a dual Banach space $E^*$ is a strong boundary;  i.e. $$C=\overline {{\rm co}(B)}^{\|\cdot\|}.$$ This result is due to Godefroy and Rod\'e and it is based upon James approach to weak compactness as an optimization problem,  see  \cite[Section 3.11.8.3]{fab-haj-mont-ziz}.

We are going to deal here with the unbounded case where only one-side James boundaries conditions are applicable in a natural way. We will follow ideas of G. Godefroy's approach, see\cite{god,god1} but combined with Theorem \ref{theo:UnboundedIEnvelopes} instead of Simon's inequality.

\begin{thm}[\textbf{Unbounded Godefroy's Theorem}]\label{TEOR:god}
Let $D \subset E^{\ast}$ be a nonempty convex \ws-compact set with $0 \notin D$ and $B \subset E^{\ast}$ a nonempty set satisfying
\begin{enumerate}
\item[(i)] For each $x \in L_{D}$ there is $b^{\ast} \in B$ with $\langle x, b^{\ast} \rangle = \sup(x,B)$.
\item[(ii)] For every convex bounded subset $L\subset E$ and every $x^{**}\in \overline{L}^{\omega^{\ast}}\subset E^{**}$ there is a  sequence $(y_n)_{n \in \N}$ in $L$ such that 
$\langle x^{**}, z^*\rangle =\lim_n \langle y_n, z^* \rangle$ for every $z^*\in B \cup D$.
\end{enumerate}
We have that
\[ \overline{{\rm co}(B)}^{\omega^{\ast}} \subset \overline{{\rm co}(B) + \Lambda_{D}}^{\|\cdot\|}. \]
Moreover, if $D$ is weakly compact then
$$\overline{{\rm co}(B)}^{\omega^{\ast}} \subset \overline{{\rm co}(B)}^{\|\cdot\|} + \Lambda_{D}.$$
\end{thm}

\begin{proof}
We will reason by contradiction, assuming that there exists an element \[ x_{0}^{\ast} \in \overline{\co{(B)}}^{\omega^{\ast}} \setminus \overline{{\rm co}(B) + \Lambda_{D}}^{\|\cdot\|}.\] 
The separation theorem provides $x^{**}_0\in B_{E^{**}}$  and $\alpha<\beta$ satisfying 
\begin{equation}\label{EQUA:auxUnboundedGodefroy}
\langle x^{**}_0,x^*\rangle < \alpha < \beta < \langle x^{**}_0, x^*_0 \rangle \text{\: for every \: $x^*\in \overline{{\rm co}(B) + \Lambda_{D}}^{\|\cdot\|}$.}
\end{equation}
In particular we have that $\langle x_{0}^{\ast \ast}, b^{\ast} + \lambda d^{\ast} \rangle < \alpha$ for every $\lambda \geq 0, b^{\ast} \in B, d^{\ast}\in D$, which implies that
\begin{equation}\label{EQUA:auxNegativeOnD} 
\langle x_{0}^{\ast \ast}, d^{\ast} \rangle \leq 0 \mbox{ \: for each \: $d^{\ast} \in D$}. 
\end{equation}
Let us define the bounded and convex set 
	\begin{equation*}
		L:=\{y\in B_E \colon \langle y, x^*_0 \rangle > \beta\}.
	\end{equation*}
By Goldstine's theorem we have that $x^{**}_0 \in \overline{L}^{\omega^{\ast}}$. Our assumption (ii) implies that there is a sequence $(y_n)_{n \in \N}$ in $L$ that converges to $x^{**}_0$ pointwise on $B \cup D$. For each $n \in \N$ we define the \ws-compact convex set 
\[ K_{n} := \{ x^{\ast} \in n B_{E^{\ast}} \cap \overline{\co{(B)}}^{\omega^{\ast}} \colon \langle y_{k}, x^{\ast} \rangle \leq \alpha \mbox{ for every $k \geq n$} \}. \]
It is clear that $B$ is contained in the union of all $K_{n}$'s because of the choice of the sequence $(y_{n})_{n \in \N}$ and inequality \eqref{EQUA:auxUnboundedGodefroy}. By Theorem \ref{theo:UnboundedIEnvelopes} we deduce that
\[ x_{0}^{\ast} \in \overline{\co{(B)}}^{\omega^{\ast}} \subseteq \overline{\co{(\mbox{$\bigcup_{n=1}^{\infty}{K_{n}}$})} + \Lambda_{D}}^{\| \cdot \|}. \]
We can then find elements $x_{1}^{\ast} \in \co{(\bigcup_{n=1}^{\infty}{K_{n}})}, \lambda_{1} \geq 0$ and $d_{1}^{\ast} \in D$ with
\begin{equation} 
\label{EQUA:PointApprox}
\| x_{0}^{\ast} - (x_{1}^{\ast} + \lambda_{1} d_{1}^{\ast}) \| < (\beta - \alpha)/2. 
\end{equation}
Using \eqref{EQUA:auxNegativeOnD} and that $(y_{n})_{n \in \N}$ converges pointwise on $D$ to $x_{0}^{\ast \ast}$, we deduce the existence of some $N \in \N$ such that for every $k > N$ 
\begin{equation} 
\label{EQUA:boundEvaluation}
\langle y_{k}, x_{1}^{\ast} + \lambda_{1} d_{1}^{\ast} \rangle \leq \alpha + \lambda_{1} \langle y_{k}, d_{1}^{\ast} \rangle < \alpha + (\beta - \alpha)/2. 
\end{equation}
Therefore, if $k >N$ then by \eqref{EQUA:PointApprox} and \eqref{EQUA:boundEvaluation} we have that
\[ \langle y_{k}, x_{0}^{\ast} \rangle \leq \| x_{0}^{\ast} - (x_{1}^{\ast} + \lambda_{1} d_{1}^{\ast}) \| + \langle y_{k}, x_{1}^{\ast} + \lambda_{1} d_{1}^{\ast} \rangle <  \beta. \]
But this contradicts the fact that the sequence $(y_{n})_{n \in \N}$ is contained in $L$.

To prove the last statement of the theorem, we just have to show that if $D$ is weakly compact then
\[  \overline{\co{(B)} + \Lambda_{D}}^{\|\cdot\|} \subset \overline{\co{(B)}}^{\|\cdot\|} + \Lambda_{D}. \]
Fix $y^{\ast}$ belonging to set on the left hand side of the previous expression. We can approximate it in norm by a sequence $x_{n}^{\ast} + \lambda_{n} d_{n}^{\ast}$ where $x_{n}^{\ast} \in \co{(B)}, \lambda_{n} \geq 0$ and $d_{n}^{\ast} \in D$ for each $n \in \N$. Notice that $(\lambda_{n})_{n \in \N}$ is bounded, since otherwise the expression
\[ \langle x_{n}^{\ast} + \lambda_{n} d_{n}^{\ast}, x_{0} \rangle \leq \sup{(x_{0}, B)} + \lambda_{n} \sup{(x_{0}, D)}, \]
would imply that $\langle x_{0},y^{\ast} \rangle  \leq - \infty $ which is absurd. Taking a subsequence, we can assume that $\lambda_{n}$ converges to some $\lambda \geq 0$. On the other hand, since $D$ is weakly compact by hypothesis we can find a subnet of $(d_{n}^{\ast})_{n \in \N}$ which weakly converges to some $d^{\ast} \in D$. The correspondent subnet of $(x_{n}^{\ast})_{n \in \N}$ must be weakly convergent to $y^{\ast} - \lambda d^{\ast}$,  and thus
\[ y^{\ast} - \lambda d^{\ast} \in \overline{\co{(B)}}^{\omega} = \overline{\co{(B)}}^{\| \cdot \|}. \]
\end{proof}

We describe now an application of the Unbounded Godefroy's theorem \ref{TEOR:god} to variational problems in dual Banach spaces. Our results here are  nonlinear analogues to those of \cite{cas-alt16,god} for weak$^{\ast}$-James compactness results. This variational setting began with \cite{ori-rui1,ori-rui2,sai} and it was used to deal with robust representation of risk measures in mathematical finance. A main contribution reads as follows:

\begin{thm}\label{dualsr}
Let $E$ be a separable Banach space without copies of $\ell^1(\N)$
and let $$f:E^* \longrightarrow \mathbb{R}\cup
\{+\infty\}$$ be a norm lower semicontinuous, convex and proper map such that
$$
\hbox{for all } x\in E, \quad x-f \hbox{ attains its
supremum on }  E^*.
$$
Then the map $f$ is \ws-lower semicontinuous and 
for every $\mu\in \mathbb{R}, $ the sublevel set 
$ f^{-1}((-\infty,\mu])$  is  \ws-compact.

\end{thm}
\begin{proof}
Let us denote with $$B= {\rm epi}(f)=\{(x^*,\xi)\in E^*\times\R: f(x^*)\leq \xi \} $$ the epigraph of $f$ in $E^*\times \R$.
The assertion  that
given $x\in E$ the function $x-f $ attains its
supremum on $E^*$ is tantamount to given $x\in E$ and $\lambda <0$ the pair $(x,\lambda)$ attains its supremum on $B$. Indeed, for every $(x,\lambda)\in E\times \R$ with $\lambda<0$ there exists $x^*_0\in \dom(f)$ such that
$$\sup\{\langle (x,\lambda), (x^*,t)\rangle \colon (x^*,t)\in B\}= \langle x, x^*_0\rangle +\lambda f(x^*_0)$$
if, and only if, the optimization problem
$$\sup\{\langle (x,-1), (x^*,t)\rangle\colon (x^*,t)\in B\}$$
which may be rewritten as
$$\sup\{\langle x, x^*\rangle -f(x^*)\colon x^*\in E^*\}$$
is attained. 
We can take a singleton set $D=\{d^*=(0,1)\}\in E^*\times \R$ this time to see that $(x,\lambda)$ attains its supremum on $B$ whenever
$\langle (x, \lambda),d^* \rangle = \lambda <0 $. Since every element $x^{\ast \ast} \in B_{E^{\ast \ast}}$ is the limit of a sequence $(x_{n})_{n \in \N}$ in $B_{E}$ (see \cite[p. 258, Theorem 5.40]{fab-haj-mont-ziz}), the hypothesis
of Theorem \ref{TEOR:god} are satisfied for the norm closed convex epigraph $B$ of $f$ and the singleton $D$, thus $B$  is going to be \ws-closed. Indeed, in this case we have $B+ \xi d^* \subset B$ for all $\xi\geq 0$ and so the conclusion of Theorem \ref{TEOR:god} reads $\overline{B}^{\omega^{\ast}}\subset  B$ as we wanted to prove, and the proof of the \ws-lower semicontinuity of $f$ is over.
Finally the level sets $ f^{-1}((-\infty,\mu])$  are bounded and so \ws-compact. Indeed, for every $x\in E$ there is $x^*_0\in E^*$
such that $$\langle x,x^*_0\rangle - f(x_0^*) \geq \langle x, y^*\rangle - f(y^*)$$ for all $y^*\in E^*$.
Thus we see that $$\langle x, y^*\rangle \leq \langle x, x^*_0\rangle -f(x^*_0) + \mu$$
for all $y^*\in f^{-1}((-\infty,\mu]) $. The Banach-Steinhaus theorem affirms now the boundness of $ f^{-1}((-\infty,\mu])$  and the proof is over.
\end{proof}

In arbitrary Banach spaces the previous theorem is not true, since the \ws-version of James' theorem for a separable Banach space $E$ holds if and only if $E$ does not contain $l^1(\N)$, see Corollary II.13 in \cite{god1}.

\section{Final remarks and questions}

\begin{enumerate}
\item It is not possible to extend Theorem \ref{TEOR:basic} for more general sets $D \subset E$. To see this, suppose that $D$ satisfies the property that \emph{whenever $B$ is a bounded subset of $E$ such that every $x^{\ast} \in L_{D}(E^{\ast}, E)$ attains its supremum on $B$, then $B$ is weakly relatively compact}. In particular, the set $B = \overline{\co{(\{ 0\} \cup D)}}^{\| \cdot \|}$ is weakly compact, since every $x^{\ast} \in L_{D}(E^{\ast}, E)$ attains its supremum on $B$ at $0$. But this implies in particular that $D$ is weakly relatively compact.

\item We do not know if Theorems \ref{TEOR:MainOneSide} and \ref{TEOR:basic} are valid for arbitrary Banach spaces. We conjecture that the answer is affirmative in both cases.

\item Let $E$ be a Banach space. If $E$ does not contain an isomorphic copy of $\ell^{1}$ then $(B_{E^{\ast}}, \omega^{\ast})$ is convex block compact (see\cite[Prop. 3.11]{bou}). This result was strengthened by Pfitzner \cite[p. 601, Proposition 11]{Pfitzner_Boundary_Problem}, who proved that if $E$ does not contain asymptotically isometric copies of $\ell^{1}$, then $(B_{E^{\ast}}, \omega^{\ast})$ is convex block compact. 

On the other hand, it can be shown that if $(B_{E^{\ast}}, \omega^{\ast})$ is block compact then $E$ does not contain a copy of $\ell^{1}(\mathbb{R})$ (see \cite{morillon} for the details). Haydon, Levy and Odell \cite{hay-lev-ode} proved that under Martin Axiom plus the negation of the Continuum Hypothesis, if $E$ does not contain a copy of $\ell^{1}(\mathbb{R})$ then $(B_{E^{\ast}}, \omega^{\ast})$ is convex block compact. In particular, this implies that convex block compactness and block compactness on $(B_{E^{\ast}}, \omega^{\ast})$ coincide under this set-theoretical assumptions. It seems to be an open question \cite[p. 270, Question 2]{morillon} whether this two notions are equivalent in general.

\item  In contrast to Theorem \ref{TEOR:GodefroyTypeTheorem2}, we cannot replace $\overline{\co{(B)} + \Lambda_{D}}^{\| \cdot \|}$ by $\overline{\co{(B)}}^{\| \cdot \|}$ in Theorem \ref{TEOR:god} as the following example shows: Consider $E = JT$ the James tree space, a dual (nonreflexive) separable Banach space without copies of $\ell^{1}$. In particular, every $x^{\ast \ast} \in B_{JT^{\ast \ast}}$ is the \ws-limit of a sequence in $B_{JT}$ (see \cite[p. 258, Theorem 5.40]{fab-haj-mont-ziz}). Hence the hypothesis (i) of Theorem \ref{TEOR:god} are satisfied for every pair of sets $B,D \subset JT^{\ast}$. Denote by $JT_{\ast}$ the predual of $JT$. Fix a norm-one element $x_{0}^{\ast} \in JT^{\ast}$ and define the sets
\[ B := \co{\left( (2 x_{0}^{\ast}+B_{JT^{\ast}}) \cup (4x_{0}^{\ast} + B_{JT_{\ast}})  \right)} \subset JT^{\ast} \]
and $D := 2 x_{0}^{\ast} + B_{JT^{\ast}}$. It is clear that every $x \in L_{D}$ attains its supremum on $B$. Therefore
\[ \overline{\co{(B)}}^{\| \cdot \|} \subsetneq \overline{\co{(B)}}^{\omega^{\ast}} \subset \overline{\co{(B)} + \Lambda_{D}}^{\| \cdot \|}. \]
\end{enumerate}



\end{document}